\documentclass[12pt]{amsart}
\usepackage{a4,graphicx}
\usepackage{color,enumerate}

\let\eps\varepsilon  
\newcommand{\N}{{\mathbb N}}  
  
\newcommand{\R}{{\mathbb R}}

\newcommand{\di}{\displaystyle}

\newcommand{\Dc}{{\mathcal D}}
\newcommand{\Fc}{{\mathcal F}}
\newcommand{\Nullb}{\mathbf{0}}

\let\di\displaystyle

\newcommand{\dv}{\operatorname{div}}

\newcommand{\charf}{ {\mbox{\large\raisebox{2pt}{$\chi$}}} }

\newcommand{\ub}{\mathbf{u}}

\newcommand{\Jb}{\mathbf{J}} 
 
\newcommand{\mub}{\boldsymbol{\mu}} 
\newcommand{\Sub}{\boldsymbol{S}} 
\newcommand{\phib}{\boldsymbol{\phi}} 
\newcommand{\wb}{\boldsymbol{w}}

\newcommand{\mukap}{\mu^{(\kappa)}}
\newcommand{\Skap}{S^{(\kappa)}}

\newcommand{\wkap}{w^{(\kappa)}}

\newcommand{\Oneb}{\mathbf{1}}
\let\pa\partial
\let\na\nabla

\newcommand{\cE}{\mathcal{E}}

\newtheorem{theorem}{Theorem}   
\newtheorem{lemma}[theorem]{Lemma}   
\newtheorem{proposition}[theorem]{Proposition}   
\newtheorem{remark}[theorem]{Remark}   
  
\newtheorem{definition}{Definition}

\begin{document}
\title{Global existence for a two-phase flow model with cross diffusion}
\author[E. S. Daus]{Esther S. Daus}
\address{Institute for Analysis and Scientific Computing, Vienna University of  
Technology, Wiedner Hauptstra\ss e 8--10, 1040 Wien, Austria}
\email{esther.daus@tuwien.ac.at}

\author[J.-P. Mili\v si\'c]{Josipa-Pina Mili\v si\'c}
\address{University of Zagreb, Faculty of Electrical Engineering and Computing, Unska 3, 10000 Zagreb, Croatia}
\email{pina.milisic@fer.hr}

\author[N. Zamponi]{Nicola Zamponi}
\address{Institute for Analysis and Scientific Computing, Vienna University of  
Technology, Wiedner Hauptstra\ss e 8--10, 1040 Wien, Austria}
\email{nicola.zamponi@tuwien.ac.at}
\date{\today}

\date{\today}
\thanks{The first and the third author acknowledge partial support from   
the Austrian Science Fund (FWF), grants P22108, P24304, W1245, P27352 and P30000.  All three authors were partially supported by the bilaterial project No.~HR 04/2018 of the Austrian Exchange Sevice OeAD together with the Ministry of Science and Education of the Republic of Croatia MZO}

\begin{abstract} 
 In this work we study a degenerate pseudo-parabolic system with cross diffusion describing the evolution of the densities of an unsaturated two-phase 
 flow mixture with dynamic capillary pressure in porous medium with saturation-dependent relaxation parameter and hypocoercive diffusion operator 
 modeling cross diffusion. The equations are derived in a thermodynamically correct way from mass conservation laws. 
 Global-in-time existence of weak solutions to the system in a bounded domain with equilibrium boundary conditions is shown. The main tools of the analysis are an entropy inequality
 and a crucial apriori bound which allows for controlling the degeneracy.
\end{abstract}



\maketitle



\section{Introduction}
   The problem of describing the transport of chemical mixtures in porous media is very important in many industrial applications.
 For a general overview on the modeling of multicomponent multiphase flows in porous media, we refer to \cite{BearBach90}.  
    In this paper we consider a two-phase flow model with wetting and non-wetting phase (e.g.~water and oil), where the non-wetting phase  consists of a mixture of $n$ chemical components, including nonequilibrium effects concerning capillary pressure and cross-diffusion effects. The main result of this work is to provide an existence analysis of the proposed model. From a mathematical viewpoint, the transport equations for the mass densities form a degenerate pseudo-parabolic system of PDEs with cross-diffusion terms. The presence of the mixed-derivative third-order term, coming from the nonequilibrium capillary pressure law, in form of a time derivative inside the diffusion operator, as well as the cross-diffusion terms, involving the chemical potentials, make the analysis very demanding. 
   Furthermore, the compactness of an approximate regularized system is obtained by applying the nonstandard compactness results of Dreyer \textit{et al.} \cite{Drey17}. 

 The modeling of nonequilibrium capillary effects in problems of enhancing oil and gas recovery from rocks was proposed by Barenblatt, Entov and Ryzhik in the classical book \cite{BER72}, and later investigated by many scientists up to nowadays. 
 In our work we follow the approach given by Hassanizadeh and Grey \cite{HG93}, where the nonequilibrium capillary effects are given by a constitutive relationship between the non-wetting phase saturation and the capillary pressure. This relationship is characterized by the presence of the relaxation parameter which depends on the non-wetting saturation as well.
 
 
Concerning the mathematical analysis,
 the global-in-time existence of weak solutions for the Richards' equation with dynamic capillary pressure and constant relaxation parameter 
 was shown by Mikeli\'c \cite{Mik10}. The first existence result for the two-phase flow model with dynamic capillary pressure
and saturation dependent relaxation parameter was obtained by Cao and Pop in \cite{CP16}.
 We note that the existence theorem can be proved under certain relations between the orders of the zeros of the relative permeabilities and the relaxation parameter and the order of the singularities of the capillary pressure function. In comparison to \cite{CP16}, here we follow the approach given in \cite{JPM17}, where it was shown that it is enough to analyze the case of the countercurrent imbibition flow instead of the full two-phase flow system. 


On the other side, the analysis of a model describing the transport of a single-phase fluid mixture in porous media taking into account also certain cross-diffusion effects was studied  in \cite{JMZ18}. The equations are derived in a thermodynamically consistent way, and global-in-time existence of weak solutions in a bounded domain with equilibrium boundary conditions as well as long-time behaviour was proved with the help of the boundedness-by-entropy method \cite{BSW12, Jue15, Jue16}. The mathematical novelties rely on the complex structure of the equations and on the observation that the solution of the binary model satisfies an unexpected integral inequality leading to a minimum principle for this system. 

Our goal in this work is to combine the strategies of \cite{JPM17} and \cite{JMZ18}, leading to a global-in-time existence of weak solutions result for a two-phase flow model with cross diffusion.

 Finally, up to our knowledge, the uniqueness and the long-time behaviour of a weak solution for a two-phase flow model with saturation-dependent relaxation parameter 
 and cross diffusion are still open problems. For a uniqueness result of a two-phase flow model with saturation-dependent relaxation parameter but without cross diffusion, we mention the result in \cite{CP15}.

\section{Model equations}\label{sec.model}

We consider an incompressible, isothermal fluid mixture with $n$ components
in a domain $\Omega \subset \R^3$. We note that the fact that we work in $\R^3$ is for convenience only, and can be easily adapted to an arbitrary space dimension $\Omega \subset \R^d$ with $d \geq 1$.  The evolution of this fluid mixture is governed by 
the transport equations for the {\em single component mass densities} $S_1(x,t),\ldots,S_n(x,t)$ in the following way
\begin{align}\label{1}
 \partial_t S_i = & \dv \left( \frac{S_i}{S} a(S) \nabla (p_c(S) + \partial_t \beta(S)) 
 + \sum_{j=1}^n D_{ij}(\Sub) \nabla \mu_j \right)\\
 &\qquad i=1,\ldots,n,~~ x\in\Omega,~~t>0.\nonumber
\end{align}
Here $S = \sum_{i=1}^n S_i$ is the {\em total mass density},
$\Sub = (S_1,\ldots,S_n)$ is the vector of the single component mass densities,
$a(S)$ is the {\em diffusion mobility},
$p_c(S)$ represents the {\em stationary capillary pressure}, 
$\tau(S)\equiv \beta'(S)$ plays the role of a relaxation parameter,
$D=(D_{ij}(\Sub))_{i,j=1,\ldots,n}$ 
is the {\em diffusion matrix},
and the quantities $\mu_1,\ldots,\mu_n$, called {\em chemical potentials}, are defined in terms of $S_1,\ldots,S_n$ as follows
\begin{align}
  \mu_i = \log \Big(  \frac{S_i}{S} \Big) + \int_{1/2}^S\frac{\tau(\sigma)}{a(\sigma)}d\sigma \qquad i=1,\ldots,n.
\label{ChemPot}
\end{align}
The sum $p_c^{\textrm{dyn}}(S)\equiv p_c(S) + \partial_t \beta(S)$ is referred to as {\em dynamic capillary pressure} \cite{HG93}.
The quantities $a(S)$, $\tau(S)$, $p_c'(S)$ are assumed to be positive for $0<S<1$, while the diffusion matrix $D(\Sub)$ is assumed to
be positive semidefinite.

Following the approach in \cite{JMZ18}, we impose equilibrium boundary conditions 
\begin{align}
 S_i & = S_i^{\Gamma}\; \textrm{  on  } \; \partial \Omega,\; t > 0,~~ i=1,\ldots,n,  \label{BC}
\end{align}
where $S_1^\Gamma,\ldots,S_n^\Gamma > 0$ are generic constants, as well as general initial conditions
\begin{align}
  S_i(\cdot,0) & = S_i^0 \; \textrm{  in } \;  \Omega,~~ i=1,\ldots,n. \label{IC}
\end{align}
For consistency of \eqref{1}, \eqref{ChemPot} with the physics, we require 
the single component concentrations $S_1,\ldots,S_n$ to be positive and the total concentration $S$ to be smaller than 1; 
that is, we seek for solutions $\Sub$ to \eqref{1}, \eqref{ChemPot} which take values in the set
\[    \Dc = \left\{  \Sub \in \R^n : \; S_i > 0 \; \textrm{ for  } \; i=1,\ldots,n, \;\; \sum_{j=1}^n S_j < 1. \right\}.  \]
The chemical potentials $\mu_1,\ldots,\mu_n$ 
are the partial derivatives
with respect to the species concentrations $S_1,\ldots,S_n$
of a free energy density function $\Fc$ satisfying
\begin{align}\nonumber
\mu_i &= \frac{\pa\Fc}{\pa S_i}(\Sub)\qquad i=1,\ldots,n,\\
 \Fc(\Sub) &= \sum_{i=1}^n S_i \log\frac{S_i}{S} + \cE(S),\qquad 
 \cE(S) = \int_{1/2}^S\int_{1/2}^{S'}\frac{\tau(\sigma)}{a(\sigma)}d\sigma d S' .
 \label{Entropy}
\end{align}
The {\em thermodynamic pressure} $p^{th}$ is given by
the Gibbs-Duhem equation
\begin{align}\label{Pth}
p^{th}(\Sub) &= \sum_{i=1}^n S_i\frac{\pa\Fc}{\pa S_i}(\Sub) - \Fc(\Sub).
\end{align}
The gradient of the thermodynamic pressure $p^{th}$ satisfies the simple relation
\begin{align}   \label{Pom}
\sum_i S_i \nabla \mu_i = \nabla p^{th}  = \frac{S\nabla\beta(S)}{a(S)} = \frac{S\tau(S)}{a(S)}\nabla S 
=\nabla \int_{1/2}^S \frac{\sigma \tau(\sigma)}{a(\sigma)}d\sigma.
\end{align}
As a consequence of \eqref{Pom}, by employing 
$\mu_i - \mu_i(\Sub^\Gamma)$ as a test function in \eqref{1}, one 
obtains the following {\em entropy balance equation}:
\begin{align}\label{ei}
 &\frac{d}{dt}\int_\Omega\left(\tilde\Fc(\Sub) +\frac{1}{2}|\nabla\beta(S)|^2 \right)dx\\ 
 &\qquad = -\int_\Omega \beta'(S)p_c'(S)|\nabla S|^2 dx - \int_\Omega\sum_{i,j=1}^n D_{ij}(\Sub)\nabla\mu_i\cdot\nabla\mu_j dx \leq 0,
 \nonumber
\end{align}
where the {\em relative entropy density} $\tilde\Fc$ 
is defined as
\begin{align}\label{rel.entr}
\tilde\Fc(\Sub) = \Fc(\Sub) - \Fc(\Sub^\Gamma) 
- \mub(\Sub^\Gamma)\cdot (\Sub - \Sub^\Gamma). 
\end{align}

\begin{remark} \textnormal{Relations \eqref{Pth}, \eqref{Pom} easily imply
\begin{align}\label{link}
\exists C\in\R : \quad 
 \sum_{i=1}^n S_i\frac{\pa\Fc}{\pa S_i}(\Sub) - \Fc(\Sub) = C + \int_{1/2}^S\frac{\sigma\tau(\sigma)}{a(\sigma)}d\sigma .
\end{align}}
\textnormal{Equation \eqref{link} constitutes a necessary condition in order for the entropy balance equation \eqref{ei} to hold; without \eqref{link}
it is unclear how to handle the contribution of the nonstationary term $\pa_t\beta(S)$ in the dynamic capillary pressure $p^{dyn}$.
In other words, \eqref{link} is a constraint on the possible choices of free energies $\Fc$ which ensure that \eqref{1} possesses an entropy structure.}

\textnormal{Since \eqref{link} is a linear nonhomogeneous equation, we can write any solution $\Fc$ to \eqref{link} as $\Fc = \Fc_0 + \Fc_1$, where $\Fc_1$
is a specific solution to \eqref{link}, while $\Fc_0$ is a generic solution to the corresponding linear homogeneous equation:
\begin{align}\label{link.0}
 \sum_{i=1}^n S_i\frac{\pa\Fc_0}{\pa S_i}(\Sub) - \Fc_0(\Sub) = 0 .
\end{align}
A simple ansatz $\Fc_1(\Sub)=\tilde\Fc_1(S)$ yields $\Fc(\Sub) = \cE(S)$ (up to additive constants). On the other hand,
Euler's theorem on homogeneous functions implies that \eqref{link.0} is equivalent to the condition that $\Fc_0$ should be homogeneous of degree
1, \textit{i.e.}~$\Fc_0(\lambda\Sub)=\lambda\Fc_0(\Sub)$ for every $\Sub\in\Dc$, $\lambda>0$. This condition has to be put together with the requirement that
$\Fc$ has to be convex and the mapping $\Sub\in\Dc\mapsto (\mu_1,\ldots,\mu_n)\in\R^n$ globally invertible. A natural choice of $\Fc_0$ which fulfills all 
these requirements is $\Fc_0(\Sub) = \sum_{i=1}^n S_i\log(S_i/S)$.}
\end{remark}
Other quantities that will play a role in the analysis of \eqref{1}
are the {\em relative chemical potentials:}
$$
\mu_i^* = \mu_i - \frac{1}{n}\sum_{j=1}^n\mu_j , \qquad 
i=1,\ldots,n.
$$
The concentrations $S_1,\ldots,S_n$ can be easily written in terms
of the total concentration and the relative chemical potentials:
\begin{align}
  S_i = S \frac{e^{\mu_i^*}}{\sum_{j=1}^n e^{\mu_j^*}},\quad i=1,\ldots,n.
 \label{Si.2}
\end{align}
The structure of the paper is as follows.
In Section \ref{sec.main} the main result of the paper is stated and the state of the art for systems of the form \eqref{1} is described.
In Section \ref{sec.aux} some auxiliary results are stated and proved.
In Section \ref{Sec.Exis} Theorem \ref{Thm.Main} is proved. 
In the Appendix the derivation of the model is shown.

\section{Main result}\label{sec.main}
Throughout the paper we make the following assumptions:
\renewcommand{\labelenumi}{\textbf{(H\theenumi)}}
\begin{enumerate}
 \item The diffusion matrix $D=(D_{ij}(\Sub))_{i,j=1}^n$ is symmetric and positive semidefinite (Onsager's principle of thermodynamics).
 Moreover, constants $D_0$, $D_1 > 0$ exist such that
\begin{align*}
D_0 \vert \Pi v \vert^2 \leq \sum_{i,j=1}^n D_{ij}(\Sub) v_i v_j \leq D_1 \vert \Pi v\vert^2 \; \textrm{  for all  } \; v \in \R^n,\; 
 \Sub \in \mathcal{D},
\end{align*}
where $\Pi = I - l \otimes l$ is the orthogonal projection on the subspace of $\R^n$ orthogonal to 
$l = (1,\ldots,1)/\sqrt{n}$.
\item The diffusion mobility $a(S)$ is given by
\begin{align*}
 a(S) = \frac{\lambda_o(S) \lambda_w(S)}{\lambda_o(S) + \lambda_w(S)} 
 = \frac{(1-S)^\lambda S^\gamma}{(1-S)^\lambda + S^\gamma}, 
 \end{align*}
for some constants $\lambda, \gamma > 0$.
\item The stationary capillary pressure $p_c(S)$ has the form
\begin{align*}
 p_c'(S) 
 = \frac{1}{S^{\beta_1}} + \frac{1}{(1-S)^{\beta_2}}, 
\end{align*}  
for some constants $\beta_1, \beta_2>0$.
\item We assume that the relaxation parameter $\tau(S)$ is given by
\begin{align*}
    \tau(S) = \beta'(S) = \frac{S^\gamma}{S^\gamma + (1-S)^\lambda} \Big[ 1 + \frac{(1-S)^\lambda}{S^{\gamma_1}}  \Big], 
\end{align*}
for some constant $\gamma_1>0$.
\item The following algebraic relations are satisfied:
\begin{align*}
5 < \beta_1 \leq \gamma_1 < \gamma 
<\frac{1}{2} \beta_1 + \frac{5}{6}(\gamma_1 -2  \big), 
\quad 5 < \beta_2 \leq \lambda < 3\beta_2-10 .
\end{align*}
\end{enumerate}

\begin{remark}
 \textnormal{ In order to avoid technical difficulties, we use explicit forms for $a$, $p_c$ and $\tau$ like in \cite{JPM17}}.
\end{remark}

\begin{remark} \textnormal{We point out that the upper bound 
$$
\sum_{i,j=1}^n D_{ij}(\Sub) v_i v_j \leq D_1 \vert \Pi v\vert^2\qquad \mbox{for }v \in \R^n,~~ \Sub \in \mathcal{D},
$$
is consistent with the fact that the diffusion fluxes $\Jb_{i} = -\sum_{j=1}^n D_{ij}(\Sub) \nabla\mu_j$ ($i=1,\ldots,n$)
sum up to zero: $\sum_{i=1}^N \Jb_i = 0$. On the other hand, the lower bound 
$$
D_0 \vert \Pi v \vert^2 \leq \sum_{i,j=1}^n D_{ij}(\Sub) v_i v_j\qquad \mbox{for }v \in \R^n,~~ \Sub \in \mathcal{D},
$$
often referred to as {\em hypocoercivity}, is the strongest coercivity property that $D$ can satisfy under the constraint
$\sum_{i=1}^N \Jb_i = 0$. As a consequence of this assumption, the diffusion fluxes $\Jb_{i}=-\sum_{j=1}^n D_{ij}(\Sub) \nabla\mu_j$ only depend
on the gradients of the relative chemical potentials: $\Jb_{i}=-\sum_{j=1}^n D_{ij}(\Sub) \nabla\mu_j^*$.}
\end{remark}
%
We now present our definition of weak solution to \eqref{1}--\eqref{IC}. In the following, the symbol 
$\langle \cdot, \cdot\rangle$ represents the duality product between
$H^{-1}(\Omega)$ and $H^{1}_0(\Omega)$.
\begin{definition}[Weak solution]\label{def.weaksol} 
A function $\Sub : \Omega\times (0,\infty)\to
\mathcal D$ is called a 
{\em global-in-time weak solution} to \eqref{1}--\eqref{IC} if and only 
if the following properties are fulfilled:
\begin{align*}
 \nabla\beta(S),~~\pa_t\beta(S),~~ \sqrt{a(S)p_c'(S)},~~ 
 \sqrt{a(S)}\nabla\pa_t\beta(S)\in L^2_{loc}(0,\infty; L^2(\Omega)),\\
 \mbox{for }i=1,\ldots,n:\quad 
 (\Pi\mu)_i\in L^2_{loc}(0,\infty; H^1(\Omega)),\\
 \mbox{for }i=1,\ldots,n:\quad 
 \pa_t S_i\in L^2_{loc}(0,\infty; H^{-1}(\Omega)),
\end{align*}
as well as the weak formulation of \eqref{1}:
\begin{align}\label{weak}
 \sum_{i=1}^n\int_0^T\langle \partial_t S_i, \phi_i\rangle dt 
 &+ \int_0^T\int_\Omega \sum_{i=1}^n\frac{S_i}{S}a(S)\left(p_c'(S)\nabla S + \nabla \partial_t \beta(S)\right)\cdot \nabla \phi_i 
 \, dx dt\\
 \nonumber
 &+ \int_0^T\int_\Omega \sum_{i,j=1}^n D_{ij}(\Sub) \nabla \mu_j \cdot \nabla \phi_i\, dx dt = 0\\
 &\nonumber \qquad \forall \phi_1,\ldots,\phi_n\in L^2_{loc}(0,\infty; H_0^1(\Omega)),
\end{align}
relation \eqref{ChemPot}, the boundary conditions \eqref{BC}\footnote{We point out that if $\beta(S)$ and $\mu_i^* = (\Pi\mu)_i$ belong
to $L^2_{loc}(0,\infty; H^1(\Omega))$ for $i=1,\ldots,b$, 
then they admit trace on 
$\pa\Omega$, therefore also $S_1,\ldots,S_n$ admit trace on 
$\pa\Omega$ thanks to the invertibility of $S\mapsto\beta(S)$ and 
relation \eqref{Si.2}.}, 
and the initial condition \eqref{IC}:
\begin{align*}
 S_i(\cdot,t)\to S_i^0\qquad \mbox{strongly in $H^{-1}(\Omega)$ as
 $t\to 0$.}
\end{align*}
\end{definition}
The result we present in this paper is concerned with the global existence of weak solutions to \eqref{1}--\eqref{IC}.
\begin{theorem}[Existence of global weak solutions]\label{Thm.Main}
Let $S_1^0,\ldots,S_n^0 : \Omega\to \R_+$ be measurable functions
satisfying 
$$ \min_{i=1,\ldots,n}\inf_\Omega S_i^0 > 0,\qquad 
\max_\Omega S^0 < 1,\qquad 
\beta(S^0)\in H^1(\Omega). $$
Assume that Assumptions \textcolor{black}{{\bf (H1)}--\bf{(H5)}} hold. Then
there exists a global-in-time weak solution $\Sub : \Omega \times (0,\infty) \to \mathcal{D}$ to \eqref{1}--\eqref{IC}.
\end{theorem}
\subsection*{Key idea of the proof}
The proof of Thr.~\ref{Thm.Main} is based on the entropy method \cite{BSW12, Jue15, Jue16}.
The starting point of the argument is the formulation of a time-discretized and regularized version of \eqref{1}. Such approximate equation
is stated in terms of the variables $w_i = \mu_i + \pa_t\beta(S)$, $i=1,\ldots,n$ (or rather a discretized version of it).
One of the key ingredients of the proof is the entropy balance equation \eqref{ei}, which yields crucial gradient estimates. The other key tool
employed in the proof is a result shown in \cite{Drey17}, which allows to prove compactness for the densities $S_1,\ldots,S_n$ if some bounds 
for the gradient of the relative chemical potentials $\nabla\mu_1^*,\ldots,\nabla\mu_n^*$ are known, together with compactness of the total
density $S$.
We point out that in the standard entropy method the approximate problem is formulated in terms of the ``entropy variables'' defined as 
partial derivatives of the mathematical entropy (or energy) density, which in the case here considered would be the functions $\mu_1,\ldots,\mu_n$ 
given by \eqref{ChemPot}. However, this standard approach does not work in this setting: in fact, 
in order to obtain a crucial estimate for the dynamic capillary pressure, $\pa_t\beta(S)$ must be used as a test function in the weak formulation
of \eqref{1}, which would clash with the regularizing terms in case these latter were written in terms of just $\mu_1,\ldots,\mu_n$.

\section{Auxiliary results}\label{sec.aux}

We present here some results which will be used in the proof of 
Thr.~\ref{Thm.Main}.
Define the variable $\boldsymbol{w}$ as follows:
\begin{align}
 \wb = \mub(\Sub) - \mub(\Sub^\Gamma) + \Big( \frac{\beta(S)-\beta(S^{k-1})}{\kappa}  \Big)\Oneb,
 \label{TotEntrVar}
\end{align} 
where we denoted $\Oneb = (1,\ldots,1)$ and 
$\Sub = (S_1,\ldots,S_n)$.
\begin{lemma}(Invertibility of $\Sub \mapsto \mub$ and 
$\Sub \mapsto \wb$)\label{lemma.3}\\
 The mappings $\Phi : \Sub\in\Dc \mapsto \mub\in\R^n$,
 and $\Phi_\kappa : \Sub\in\Dc \mapsto \wb\in\R^n$
 are invertible, and their Jacobians $\Phi'$, $\Phi_\kappa'$ are uniformely positive definite in $\Dc$.
\end{lemma}
\begin{proof} We note that $\di \frac{\partial \mu_i}{\partial S_j} = \frac{\partial^2 \Fc}{\partial S_j \partial S_i}$. Direct calculation gives that 
\[ [\Fc'']_{ij} = \frac{\partial^2 \Fc}{\partial S_j \partial S_i} 
  = \begin{cases}
  \frac{1}{1-S}, & i\neq j,\\
  \frac{1}{S_i} + {\frac{1}{(1-S)}}, & i=j,
  \end{cases} \]
  from where it follows that $\Fc''$ is uniformly positive definite in $\Dc$, \textit{i.e.} $\Fc : \Dc\to\R$ is a differentiable, strictly convex mapping. 
As a consequence, its gradient $\Phi = \Fc' : \Dc\to\R^n$ is a monotone (and therefore injective) mapping. 
Its inverse $\Phi^{-1}$ can be explicitly computed: {$\Phi^{-1}(\mub)_i = \frac{e^{\mu_i}}{1+\sum_{j=1}^n e^{\mu_j}}$, $i=1,\ldots,n$, $\mub\in\R^n$.}
Therefore $\Phi : \Dc\to\R^n$ is invertible. Moreover, since 
$\beta' \geq 0$, then $\frac{\partial\wb}{\partial\Sub}$ is symmetric and
positive definite. Furthermore, $\lim_{\Sub\to\pa\Dc}|w(\Sub)| = \infty$. Using the Hadamard global inverse theorem, \cite[Thm. 2.2]{RS15}, we conclude that $\Phi_\kappa : \Dc\to\R^n$ is invertible.
\end{proof}
%
\begin{lemma}\label{lem.symm}
Let $f : [0,1]\to\R$ be a continuous function with $f'(S)>0$ for $S \in (0,1)$.
Given any $\wb\in\R^n$, we denote by $\Sub=\Sub(\wb)\in\{\Sub\in (0,\infty)^n ~:~ \sum_{i=1}^n S_i < 1\}$ the only solution to
$$ w_i = \log(S_i) - \log(S) + f(S),~~ i=1,\ldots,n ,\qquad S\equiv \sum_{i=1}^n S_i . $$
Then the matrix $(M_{ij}(\wb))_{i,j=1,\ldots,n}=(S_i(\wb)\frac{\pa S(\wb)}{\pa w_j})_{i,j=1,\ldots,n}$ is symmetric and positive semidefinite
for every $\wb\in\R^n$.
\end{lemma}
\begin{proof}
The definition of $\Sub$ implies
\begin{align*}
 \sum_{i=1}^n e^{w_i} = e^{f(S(\wb))}.
\end{align*}
Differentiating the above identity with respect to $w_j$ leads to
\begin{align*}
e^{w_j} = F(S(\wb))\frac{\pa S(\wb)}{\pa w_j},\qquad F(S)\equiv \frac{d}{dS}\left(e^{f(S)}\right) =  e^{f(S)} f'(S).
\end{align*}
Since $f$ is strictly increasing, then $F>0$ in $(0,1)$. It follows
\begin{align*}
 M_{ij}(\wb) = \frac{S_i(\wb)e^{w_j}}{F(S(\wb))} = \frac{e^{f(S(\wb))}S_i(\wb)S_j(\wb)}{S(\wb)F(S(\wb))} ,
\end{align*}
which means that $M(\wb)$ is symmetric and positive semidefinite for every $\wb\in\R^n$.
This finishes the proof.
\end{proof}
%
%
%
\begin{lemma}
 The following bound holds
  \begin{equation}
    p_c^\prime (S) \leq \frac{\tau(S)}{a(S)},\quad S\in (0,1).
    \label{coeff:ineq}
  \end{equation}
  \label{lemma:coeff:1}
\end{lemma}
\begin{proof}
  Through simple calculations using Assumptions \textcolor{black}{\bf{(H2)}--\bf{(H5)}}, (\ref{coeff:ineq}) can be written as 
  \begin{align*}
    \frac{1}{S^{\beta_1}} + \frac{1}{(1-S)^{\beta_2}} \leq \frac{1}{S^{\gamma_1}} + \frac{1}{(1-S)^{\lambda}} .
  \end{align*}
  Since $S\in (0,1)$, the claim follows from the fact that 
  $\beta_1\leq\gamma_1$, $\beta_2\leq\lambda$.
\end{proof}
The next result has been proved in \cite[Lemma 5]{JMZ18}:
\begin{lemma}\label{lem.JMZ}
 Let $\boldsymbol{\alpha}$, $\boldsymbol{\beta} \in \R^n$ be such that $|\boldsymbol{\alpha}| = |\boldsymbol{\beta}| = 1$. Then, for any $\boldsymbol{v} \in \R^n$ it holds that
\[ |\boldsymbol{\alpha} \cdot \boldsymbol{v}|^2 + |\boldsymbol{v} -(\boldsymbol{\beta} \cdot \boldsymbol{v})\boldsymbol{\beta}|^2 \geq \frac{1}{4}(\boldsymbol{\alpha} \cdot \boldsymbol{\beta})^2|\boldsymbol{v}|^2 . \]
\end{lemma}

\begin{flushleft}
{\bf Notation.} Let $\R_+\equiv [0,\infty)$.
For $x \in \R_+\times \R^{N-1}$, we denote $x = (x_0,\overline{x})$. 
\end{flushleft}
\begin{lemma}\label{lemma.13}
 Let $\mathcal{R} : \R_+\times \R^{N-1}\to \R_+^N$ be a continuous and bounded mapping.
 Let $K \subset L^2(\Omega)$ be relatively compact. 
 Let $\{ \phi_i \in C_c^\infty(\Omega;\R^N): i\in\N\}$ be dense in $L^2(\Omega;\R^N)$. 
 Then, for every $\delta > 0$, there are $C(\delta) > 0$, $m(\delta) \in \N$ such that,
 for all $\boldsymbol{w}^1,\boldsymbol{w}^2 \in K\times H^1(\Omega; \R^{N-1})$ it holds
 \begin{align}\nonumber
  \| \mathcal{R}(\boldsymbol{w^1}) & - \mathcal{R}(\boldsymbol{w^2})\|_{L^2(\Omega)} \\
  & \leq \delta \Big(  1 + \sum_{i=1,2} \|\overline{\boldsymbol{w}}^i\|_{H^1(\Omega)} \Big) 
  + C(\delta) \sum_{i=1}^{m(\delta)} \Big| \int_\Omega \big(\mathcal{R}(\boldsymbol{w}^1) - \mathcal{R}(\boldsymbol{w}^2) \big)\cdot \phi_i dx \Big| .
 \label{Dreyer.1}
 \end{align}
\end{lemma}
\begin{proof} Assume by contradiction that there exists $\delta_0>0$ such that, for every $m\in\N$, there exist 
$\boldsymbol{w}^{1,m}, \boldsymbol{w}^{2,m}\in K\times H^1(\Omega ; \R^{N-1})$ such that 
\begin{align*}
 \|\mathcal{R}(\boldsymbol{w}^{1,m}) & - \mathcal{R}(\boldsymbol{w}^{2,m})\|_{L^2(\Omega)}\\ 
 & > \delta_0\Big(  1 + \sum_{i=1,2} \|\overline{\boldsymbol{w}}^{i,m}\|_{W^{1,1}(\Omega)} \Big) 
  + m \sum_{i=1}^{n} \Big| \int_\Omega \big(\mathcal{R}(\boldsymbol{w}^{1,m}) - \mathcal{R}(\boldsymbol{w}^{2,m}) \big)\cdot \phi_i dx \Big| .
\end{align*}
Since $\mathcal{R}(\R_+\times \R^{N-1})$ is bounded, then 
$(\overline{\boldsymbol{w}}^{i,m})_{m\in\N}$ 
is bounded in 
$H^1(\Omega ; \R^{N-1})$ and thus 
$\overline{\boldsymbol{w}}^{i,m}\rightharpoonup\overline{\boldsymbol{w}}^i$ 
weakly in $H^1(\Omega ; \R^{N-1})$
(as $m\to\infty$), for $i=1,2$. By a compact Sobolev embedding it holds that $\overline{\boldsymbol{w}}^{i,m}\to \overline{\boldsymbol{w}}^{i}$ strongly in $L^2(\Omega)$ and 
a.e.~in $\Omega$ (up to a subsequence), for $i=1,2$. 
Moreover, the compactness of $K$ implies that $\boldsymbol{w}_0^{i,m}\to \boldsymbol{w}_0^i$ strongly in $L^2(\Omega)$ (up to a subsequence), for $i=1,2$.
Therefore, $\boldsymbol{w}^{i,m}\to \boldsymbol{w}^{i}$ strongly in $L^2(\Omega ; \R^N)$ and a.e.~in $\Omega$. It follows that 
$\mathcal{R}(\boldsymbol{w}^{i,m})\to \mathcal{R}(\boldsymbol{w}^i)$ strongly in $L^2(\Omega ; \R^N)$. On the other hand,
\begin{align*}
 &\sum_{i=1}^{n} \Big| \int_\Omega \big(\mathcal{R}(\boldsymbol{w}^{1,m}) - \mathcal{R}(\boldsymbol{w}^{2,m}) \big)\cdot \phi_i dx \Big|\\
 & \qquad 
 \leq \frac{1}{m}\|\mathcal{R}(\boldsymbol{w}^{1,m}) - \mathcal{R}(\boldsymbol{w}^{2,m})\|_{L^2(\Omega)}\leq \frac{C}{m}\to 0\quad (m\to\infty),
\end{align*}
and so
\begin{align*}
 \int_\Omega \big(\mathcal{R}(\boldsymbol{w}^{1}) - \mathcal{R}(\boldsymbol{w}^{2}) \big)\cdot \phi_i dx = 0\quad\forall i\in\N .
\end{align*}
Being $(\phi_i)_{i\in\N}$ dense in $L^2(\Omega)$, this implies that $\mathcal{R}(w^1) = \mathcal{R}(w^2)$. But
\begin{align*}
 \|\mathcal{R}(\boldsymbol{w}^1)-\mathcal{R}(\boldsymbol{w}^2)\|_{L^2(\Omega)} = \lim_{m\to\infty}\|\mathcal{R}(\boldsymbol{w}^{1,m}) - \mathcal{R}(\boldsymbol{w}^{2,m}) \|_{L^2(\Omega)}
 \geq\delta_0 > 0,
\end{align*}
which is a contradiction.
 This finishes the proof.
\end{proof}

We recall the following remark, see \cite{Drey17}. For completeness and clarity, we give a full proof.
\begin{lemma}\label{remark.14}
If a subset $\{ u_{\eps} \}_{\eps\in (0,1]}$ of $C([0,T];L^2(\Omega))$ is relatively compact in \\
$C([0,T];L^2(\Omega))$, then the set 
$\mathcal F \equiv \cup_{\eps\in (0,1]} \cup_{t\in[0,T]}\{u_\eps(t)\}$ is relatively compact in $L^2(\Omega)$. In this case, given any 
$f\in C^0(\R)$, the set $\mathcal F_f \equiv \cup_{\eps\in (0,1]} \cup_{t\in[0,T]}\{f(u_\eps(t))\}$ is relatively compact in $L^2(\Omega)$.
\end{lemma}
\begin{proof}
Let $(u_{\eps_n}(t_n))_{n\in\N}$ be an arbitrary sequence of points of $\mathcal F$. The sequence $(u_{\eps_n})_{n\in\N}\subset C([0,T]; L^2(\Omega))$
is relatively compact in $C([0,T]; L^2(\Omega))$, therefore is convergent up to a subsequence. Moreover, the sequence $(t_n)_{n\in\N}\subset [0,T]$
is convergent up to a subsequence, so w.l.o.g. we can write $u_{\eps_n}\to u$ strongly in $C([0,T]; L^2(\Omega))$ and $t_n\to t\in [0,T]$. It follows that
\begin{align*}
\|u_{\eps_n}(t_n) - u(t)\|_{L^2(\Omega)} 
&\leq \|u_{\eps_n}(t_n) - u(t_n)\|_{L^2(\Omega)} + \|u(t_n) - u(t)\|_{L^2(\Omega)} \\
&\leq \|u_{\eps_n} - u\|_{C([0,T]; L^2(\Omega))} + \|u(t_n) - u(t)\|_{L^2(\Omega)} 
\quad\substack{\phantom{aaa} \\ \longrightarrow\\ n\to\infty}\quad 0 .
\end{align*}
Therefore $\mathcal{F}$ is relatively compact in $L^2(\Omega)$. In this case, given any $f\in C^0(\R)$, the relative compactness of $\mathcal{F}_f$
in $L^2(\Omega)$ is straightforward. This finishes the proof of the Lemma.
\end{proof}
Our main compactness tool is given in the following lemma (see Corollary 3.7. in \cite{Drey17}).
\begin{lemma}\label{lemma.15}
 For $n \in \N$, let $w^n:[0,T] \to L^2(\Omega;\R_+\times \R^{N-1})$ be continuous. 
 Assume that $K := \{ w_0^n(\cdot,t)\in L^2(\Omega) : n\in\N, t \in [0,T] \}$ is relatively compact in $L^2(\Omega)$, and that $\overline{w}^n$ is 
 bounded in $L^1((0,T); H^1(\Omega))$. Furthermore, let $\mathcal{R} : \R_+ \times \R^{N-1} \to \R_+^N$ be continuous and bounded. Then,
 $\mathcal{R}(w^n)$ is (up to subsequence) strongly convergent in $L^1(\Omega\times (0,T))$.
\label{Dreyer.2}
\end{lemma}
\begin{proof}
Apply Lemma~\ref{lemma.13}. For every $\delta>0$ there exist $C(\delta)>0$, $m(\delta)\in\N$ such that, for every $n,n'\in\N$ it holds that
\begin{align*}
 &\|\mathcal{R}(w^n(t)) - \mathcal{R}(w^{n'}(t))\|_{L^2(\Omega)}
 \leq \delta(1 + \|\overline{w}^n(t)\|_{H^1} + \|\overline{w}^{n'}(t)\|_{H^1})\\
 &\qquad +C(\delta)\sum_{i=1}^{m(\delta)}\left| \int_{\Omega}\left(\mathcal{R}(w^n(t)) - \mathcal{R}(w^{n'}(t))\right)\cdot\phi_i dx\right| .
\end{align*}
By integrating the above estimate in time and exploiting the boundedness of $\overline{w}^n$ in $L^1((0,T); H^1(\Omega))$, we deduce
\begin{align*}
 &\int_0^T\|\mathcal{R}(w^n(t)) - \mathcal{R}(w^{n'}(t))\|_{L^2(\Omega)}dt\\
 &\leq \delta C
 +C(\delta)\sum_{i=1}^{m(\delta)}\int_0^T\left| \int_{\Omega}\left(\mathcal{R}(w^n(t)) - \mathcal{R}(w^{n'}(t))\right)\cdot\phi_i dx\right|dt .
\end{align*}
The boundedness of the mapping $\mathcal{R}$ implies that, up to subsequences, 
$\mathcal{R}(w^n(t))$ is weakly convergent in $L^2(\Omega)$ for a.e.~$t\in [0,T]$, and so
\begin{align*}
 \int_{\Omega}\left(\mathcal{R}(w^n(t)) - \mathcal{R}(w^{n'}(t))\right)\cdot\phi_i dx\to 0\quad\mbox{as }n, n'\to\infty ,~~\mbox{a.e. }t\in [0,T], 
 \quad i\in\N .
\end{align*}
Moreover, 
\begin{align*}
 \left| \int_{\Omega}\left(\mathcal{R}(w^n(t)) - \mathcal{R}(w^{n'}(t))\right)\cdot\phi_i dx\right|\leq C\|\phi_i\|_{L^2(\Omega)},~~
 \mbox{a.e. }t\in [0,T],~~ i\in\N .
\end{align*}
The dominated convergence theorem yields
\begin{align*}
 \int_0^T\left| \int_{\Omega}\left(\mathcal{R}(w^n(t)) - \mathcal{R}(w^{n'}(t))\right)\cdot\phi_i dx\right|dt\to 0\quad\mbox{as }n, n'\to\infty,
 ~~i\in\N .
\end{align*}
It follows that $\nu\in\N$ exists such that, for $n,n'\geq\nu$,
\begin{align*}
 \int_0^T\left| \int_{\Omega}\left(\mathcal{R}(w^n(t)) - \mathcal{R}(w^{n'}(t))\right)\cdot\phi_i dx\right|dt\leq\frac{\delta}{m(\delta)C(\delta)},
 ~~1\leq i\leq m(\delta) .
\end{align*}
As a consequence, it holds that
\begin{align*}
 \int_0^T\|\mathcal{R}(w^n(t)) - \mathcal{R}(w^{n'}(t))\|_{L^2(\Omega)}dt\leq \delta C, \quad n,n'\geq\nu .
\end{align*}
In particular, $\mathcal{R}(w^n)$ is Cauchy (and therefore convergent) in $L^1(\Omega\times (0,T))$. This finishes the proof.
\end{proof}



\section{Existence proof}\label{Sec.Exis}

The proof is divided into several steps.\medskip

{\bf Step 1: discretization and regularization.}  
Fix $T>0$. For $N\in \N$ we define $\kappa = T/N$, $t_k = \kappa k$ ($k=0,\ldots,N$), $S_{i}^{0} = S_{i,0}$ $(i=1,\ldots,n)$. 

Consider the implicit Euler  discretization:
\begin{align}\nonumber
&\mbox{given ${\wb^{k-1}}\in H^1_0(\Omega ; \R^n)$,
 find ${\wb^k}\in H^1_0(\Omega; \R^n)$ such that:}\\
& \sum_{i=1}^n  \int_\Omega  \frac{S_i^k - S_i^{k-1}}{\kappa}\phi_i dx 
 =  -\sum_{i=1}^n \int_\Omega \frac{S_i^k}{S^k} a(S^k)  p_c'(S^k) \nabla S^k \cdot \nabla \phi_i dx \nonumber\\
 &  \qquad  -  \sum_{i=1}^n \int_\Omega \frac{S_i^k}{S^k} a(S^k)   \nabla \frac{\beta(S^k) - \beta(S^{k-1}) }{\kappa}  \cdot \nabla \phi_i dx \nonumber \\
 & \qquad - \int_\Omega \sum_{i,j=1}^n D_{ij}(S_1^k,\ldots,S_n^k)\nabla w_j^k \cdot \nabla \phi_i dx \nonumber\\
 & \qquad - \eps  \sum_{i=1}^n \int_\Omega  \frac{S_i^k}{S^k} \nabla w_i^k\cdot  \nabla \phi_i dx, 
\label{DisReg.1}
\end{align}
for all $\phi_1,\ldots,\phi_n \in H^1_0(\Omega)$,
where $\Sub^k : \Omega\times (0,T)\to\R^n$ is (implicitly) defined by 
\[ w_i^k = \log \Big(  \frac{S_i^k}{S^k}\Big) 
+ \frac{\beta(S^k) - \beta(S^{k-1})}{\kappa}, \qquad i=1,\ldots,n, \]
and we denoted $S^k = \sum_{i=1}^n S_i^k$. 
Here we assume that $S^{k-1}\in H^{1}(\Omega)$.

\medskip

{\bf Step 2: linearized approximated problem.}  
Using the fact that 
\[ \nabla S = \sum_{\ell=1}^n\frac{\pa S}{\pa w_\ell}{(w)}\na w_\ell, \]
equation  \eqref{DisReg.1} can be simply rewritten as
 \begin{align}\nonumber
&  \sum_{i=1}^n \int_\Omega  \frac{S_i^k - S_i^{k-1}}{\kappa}\phi_i dx 
 =  - \sum_{i=1}^n \int_\Omega \frac{S_i^k}{S^k} a(S^k)  p_c'(S^k) \sum_{\ell=1}^n\frac{\pa S}{\pa w_\ell}{(w^k)}\na w_\ell^k\cdot \nabla \phi_i dx \nonumber\\
 &  \qquad  - \frac{1}{\kappa}\sum_{i=1}^n \int_\Omega \frac{S_i^k}{S^k} a(S^k) 
\tau(S^k)\sum_{\ell=1}^n\frac{\pa S}{\pa w_\ell}\na w_\ell^k\cdot \nabla \phi_i dx \nonumber\\
&  \qquad   + \frac{1}{\kappa}\sum_{i=1}^n\int_\Omega \frac{S_i^k}{S^k} a(S^k) \tau(S^{k-1})\nabla S^{k-1} \cdot \phi_i dx \nonumber \\
 & \qquad - \int_\Omega \sum_{i,j=1}^n D_{ij}(S_1^k,\ldots,S_n^k)\nabla w_j^k \cdot \nabla \phi_i dx \nonumber\\
 &\qquad - \eps  \sum_{i=1}^n \int_\Omega  \frac{S_i^k}{S^k} \nabla w_i^k\cdot  \nabla \phi_i dx. 
\label{DisLin.2}
\end{align}
 Now, the linearized problem has the following form:
\begin{align}\nonumber
&\mbox{let $\wb^*\in L^2(\Omega)$ and  $\sigma\in [0,1]$ be given, 
 find $\wb\in H^1_0(\Omega)$ such that:}\\
&  \sigma\sum_{i=1}^n \int_\Omega  \frac{S_i^* - S_i^{k-1}}{\kappa}\phi_i dx \nonumber\\
&\quad =  - \sum_{i=1}^n \int_\Omega \frac{S_i^*}{S^*} a(S^*)  p_c'(S^*) \sum_{\ell=1}^n\frac{\pa S}{\pa w_\ell}(w^*)\na w_\ell\cdot \nabla \phi_i dx \nonumber\\
 &  \qquad  - \frac{1}{\kappa}\sum_{i=1}^n \int_\Omega \frac{S_i^*}{S^*} a(S^*) 
\tau(S^*)\sum_{\ell=1}^n\frac{\pa S}{\pa w_\ell}(w^*)\na w_\ell\cdot \nabla \phi_i dx \nonumber\\
&  \qquad   + \frac{\sigma}{\kappa}\sum_{i=1}^n\int_\Omega \frac{S_i^*}{S^*} a(S^*) \tau(S^{k-1})\nabla S^{k-1} \cdot {\nabla \phi_i }dx \nonumber \\
 & \qquad - \int_\Omega \sum_{i,j=1}^n D_{ij}(S_1^*,\ldots,S_n^*)\nabla w_j \cdot \nabla \phi_i dx \nonumber\\
 &\qquad - \eps  \sum_{i=1}^n \int_\Omega  \frac{S_i^*}{S^*} \nabla w_i\cdot  \nabla \phi_i dx, 
\label{DisRegLin.1}
\end{align}
for all $\phi_i \in H_0^1(\Omega)$, where $S_i^*$ is defined by
\[ w_i^* = \log \Big(  \frac{S_i^*}{S^*}\Big) 
+ \frac{\beta(S^*) - \beta(S^{k-1})}{\kappa}, \qquad i=1,\ldots,n, \]
and we denoted $S^* = \sum_{i=1}^n S_i^*$. The above problem can be summarized as
\begin{align}
   a(\wb,\phib) = \sigma F(\phib),\quad \forall \phib \in H_0^1(\Omega;\R^n),    
   \label{LaxMil_Eq}
   \end{align}
where
\begin{align}\nonumber
a(\wb,\phib) & = \sum_{i,\ell=1}^n \int_\Omega  \frac{S_i^*}{S^*}\frac{\pa S}{\pa w_\ell}{(S^*)} a(S^*)\Big[  p_c'(S^*) + \frac{\tau(S^*)}{\kappa}  \Big] 
\nabla w_\ell \cdot \nabla \phi_i dx \\
& + \sum_{i,j=1}^n \int_\Omega   D_{ij}(S_1^*,\ldots,S_n^*)  \nabla w_j \cdot \nabla \phi_i dx \nonumber \\
& {+} \eps  \sum_{i=1}^n \int_\Omega  \frac{S_i^*}{S^*} \nabla w_i\cdot  \nabla \phi_i dx
\label{bili_form}
\end{align}
\begin{align}
F(\phib) & = -   \sum_{i=1}^n \int_\Omega  \frac{S_i^* - S_i^{k-1}}{\kappa}\phi_i dx 
 - \int_\Omega \frac{S_i^*}{S^*} a(S^*) \frac{\tau(S^{k-1})\nabla S^{k-1}}{\kappa} \cdot {\nabla \phi_i} dx.
\label{Funct} 
\end{align}
It is easy to see that the functional $F$ is continuous, \textit{i.e.} it holds
\[ |F(\phib)| \leq C \|\phib\|_{H^1(\Omega,\R^n)}. \]
The bilinear form \eqref{bili_form} can be written as:
\begin{align*}
  a(\wb,\phib) & = \sum_{i,j=1}^n \int_\Omega  \alpha_{ij}(S_1^*,\ldots,S_n^*) \nabla {w_j} \cdot \nabla \phi_i dx 
   +  \eps  \sum_{i=1}^n \int_\Omega  \frac{S_i^*}{S^*} \nabla w_i\cdot  \nabla \phi_i dx,
\end{align*}
with
\[  \alpha_{ij}(S_1^*,\ldots,S_n^*) =  D_{ij}(S_1^*,\ldots,S_n^*)  + S_i^* \frac{\pa S}{\pa {w_j}}(S^*) G(S^*), \]
where
\[ G(S^*) = \frac{a(S^*)}{S^*} \Big[  p_c'(S^*) + \frac{\tau(S^*)}{\kappa}  \Big]. \]

Thanks to Lemma \ref{lem.symm} and the nonnegativity of $G(S^*)$:
\begin{align*}
 a(\wb,\wb) &\geq  
 \sum_{i,j=1}^n \int_\Omega  D_{ij}(S_1^*,\ldots,S_n^*) \nabla  w_i \cdot \nabla w_j dx 
 +  \eps\sum_{i=1}^n\int_\Omega\frac{S_i^*}{S^*} |\nabla w_i|^2 dx .
\end{align*}
From Assumption \textbf{(H1)} we obtain
\begin{align*}
 & a(\wb,\wb) \\ 
 &\geq  
 \min\{D_0,\eps\}\left( \sum_{i,j=1}^n \int_\Omega  \left(\delta_{ij}-\frac{1}{n}\right) \nabla  w_i \cdot \nabla w_j dx 
 +  \sum_{i=1}^n \int_\Omega \left|\sqrt{\frac{S_i^*}{S^*}}\nabla w_i\right|^2 dx \right).
\end{align*}

 Now we apply Lemma \ref{lem.JMZ} and deduce
 \begin{align} \nonumber
 a(\wb,\wb) 
 & \geq 
 \frac{\min(D_0,\eps)}{4 n}  \int_\Omega \Big( \sum_{i=1}^n \sqrt{\frac{S_i^*}{S^*}} \Big)^2 |\na \wb|^2 dx .
 \nonumber
 \end{align}
Next, since
$ \Big( \sum_{i=1}^n \sqrt{\frac{S_i^*}{S^*}} \Big)^2 \geq n$,
we conclude that the bilinear form $a(\wb, \wb)$ is coercive
in $H^1_0(\Omega)$, \textit{i.e.}
\begin{align}\nonumber
a(\wb, \wb) &\geq\sum_{i,j=1}^n\int_\Omega D_{ij}(S_1^*,\ldots,S_n^*)\nabla  w_i\cdot \nabla w_j dx + \eps\sum_{i=1}^n\int_\Omega\frac{S_i^*}{S^*}|\nabla w_i|^2 dx\\
\label{lin.lb}
 &\geq C(\eps) \|\nabla \wb\|_{L^2(\Omega,\R^n)}^2 \geq 
 C(\eps)\|\wb\|_{H^1(\Omega,\R^n)}^2 ,
\end{align}
the last inequality being a consequence of Poincar\'e's Lemma.
Therefore we can deduce by Lax-Milgram lemma the existence of a unique solution $\wb \in H^1_0(\Omega;\R^n)$ to \eqref{DisRegLin.1}.
\begin{remark} \label{Bound-nabla-mu}
We note that from the coercivity of the bilinear form $a(\wb,\wb)$ it directly follows that the solution $\wb\in H^1_0(\Omega)$
to the linearized problem satisfies 
$\| \nabla \wb\|_{L^2(\Omega, \R^n)} \leq C(\eps)$.
\end{remark}
{\bf Step 3: solution of the nonlinear approximated problem.} We reformulate \eqref{DisReg.1} as a fixed-point problem for a suitable operator and we solve it via Leray-Schauder fixed point theorem.
The {\bf Step 2} allows us to define an operator $T: L^2(\Omega;\R^n)\times [0,1] \to L^2(\Omega;\R^n)$ in the following way:
for $\wb^{*} \in L^2(\Omega;\R^n)$, $\sigma \in [0,1]$, it holds that $\wb = T(\wb^{*},\sigma) \in H^1_0(\Omega;\R^n)$ is the solution to \eqref{DisRegLin.1}. In a standard way we can show that the mapping $T$ is continuous. Moreover, $ T: L^2(\Omega;\R^n)\times [0,1]  \to L^2(\Omega;\R^n)$ is compact due to the compact Sobolev embedding $H^1(\Omega;\R^n) \hookrightarrow L^2(\Omega;\R^n)$. 
Furthermore, it holds that $T(\cdot,0)\equiv 0$.
It remains to prove a uniform bound (with respect to $\sigma$) for all fixed points of $T(\cdot,\sigma)$ in $L^2(\Omega,\R^n)$. Let $\wb \in L^2(\Omega,\R^n)$ be such a fixed point. Then $\wb$ solves \eqref{LaxMil_Eq} with a test-function $\phi$ replaced by $\wb$.
We have
\begin{align*}
 C(\eps) \|\wb\|_{H^1(\Omega,\R^n)}^2 \leq a(\wb,\wb) = \sigma F(\wb) \leq C \|\wb\|_{L^2(\Omega,\R^n)},
\end{align*} 
yielding an $H^1$ bound for $\wb$, uniform in $\sigma$. 
Thanks to Leray-Schauder's fixed point theorem we get the existence of a solution $\wb \in H^1_0(\Omega;\R^n)$ to
\eqref{LaxMil_Eq} for $\sigma = 1$. In this way we proved the solution to \eqref{DisReg.1}.

{\bf Step 4: uniform in $\kappa$ a-priori estimates.}  
Let us choose
$$\phi  = \wb^k = \mub^k - \mub(\Sub^\Gamma) +\Big(  \frac{\beta(S^k)-\beta(S^{k-1})}{\kappa}  \Big) \Oneb $$ 
in \eqref{DisReg.1}.
Since $\mu_i^k = \partial_i \Fc(\Sub^k)$ and $\Fc(\Sub)$ is convex, it follows that
 \[   \sum_{i=1}^n (S_i^k - S_i^{k-1})(\mu_i^k - \mu_i(\Sub^\Gamma))
 \geq \tilde\Fc(\Sub^k) - \tilde\Fc(\Sub^{k-1}),  \]
where $\tilde\Fc$ is the relative entropy density defined in 
\eqref{rel.entr}.
Moreover, the nonnegativity and boundedness of $\beta'$ allows us to write
$$
(S^k - S^{k-1})(\beta(S^k) - \beta(S^{k-1})) \geq C( \beta(S^k) - \beta(S^{k-1}) )^2,
$$
where $\di C = \frac{1}{\max_{0 \leq S \leq 1} \beta'(S)}$.
In this way we obtain 
\begin{align}\nonumber
&\frac{1}{\kappa}\int_\Omega \tilde\Fc(\Sub^k) dx + C \int_\Omega\left(\frac{\beta(S^k)-\beta(S^{k-1})}{\kappa}\right)^2 dx\\
\nonumber
&\qquad + \sum_{i=1}^n
\int_\Omega \frac{S_i^k}{S^k}a(S^k) p_c'(S^k) \nabla S^k \cdot \nabla w_i^k dx \nonumber\\
& \qquad  + \frac{1}{\kappa} \sum_{i=1}^n \int_\Omega \frac{S_i^k}{S^k} a(S^k) \big( \nabla \beta(S^k) - \nabla \beta(S^{k-1}) \big) \cdot \nabla w_i^k dx\nonumber \\
& \qquad + \sum_{i,j=1}^n \int_\Omega D_{ij}(S_1^k,\ldots,S_n^k)\nabla w_j^k \cdot \nabla w_i^k dx 
+ {\eps \sum_{i=1}^n \int_{\Omega}\frac{S_i^k}{S^k} |\nabla w_i^k|^2 dx} \nonumber \\
& \qquad \qquad \leq \frac{1}{\kappa} \int_{\Omega} \tilde\Fc(\Sub^{k-1})  dx.  
\label{TF1-1}
\end{align}      
Taking into account \eqref{lin.lb} and Assumption \textbf{(H1)}, one gets
\begin{align}\nonumber
&\frac{1}{\kappa}\int_\Omega \tilde\Fc(\Sub^k) dx 
+ C \int_\Omega\left(\frac{\beta(S^k)-\beta(S^{k-1})}{\kappa}\right)^2 dx\\
\nonumber
&\qquad + \sum_{i=1}^n
\int_\Omega \frac{S_i^k}{S^k}a(S^k) p_c'(S^k) \nabla S^k \cdot \nabla w_i^k dx \nonumber\\
& \qquad  + \frac{1}{\kappa} \sum_{i=1}^n \int_\Omega \frac{S_i^k}{S^k} a(S^k) \big( \nabla \beta(S^k) - \nabla \beta(S^{k-1}) \big) \cdot \nabla w_i^k dx\nonumber \\
&\qquad + C\|\Pi\na {w^k}\|_{L^2(\Omega)}^2
\nonumber\\
& \qquad + C\eps\|w\|_{H^1(\Omega)}^2 \leq \frac{1}{\kappa} \int_\Omega \tilde\Fc(\Sub^{k-1}) dx.  
\label{TF1-2}
\end{align}      
Using the relation \eqref{Pom},
 we obtain
\begin{align}
\sum_{i=1}^n
\int_\Omega \frac{S_i^k}{S^k}a(S^k) &  p_c'(S^k) \nabla S^k \cdot \nabla \mu_i^k dx 
     = \int_\Omega  p_c'(S^k) \beta'(S^k) |\nabla S^k|^2 dx.
\label{TF1-3}
\end{align}
In this way we get:
\begin{align}\nonumber
\frac{1}{\kappa} \int_\Omega \tilde\Fc(\Sub^k)  dx &
+ C \int_\Omega\left(\frac{\beta(S^k)-\beta(S^{k-1})}{\kappa}\right)^2 dx + \int_\Omega p_c'(S^k) \beta'(S^k) |\nabla S^k|^2 dx  \nonumber\\
& + \frac{1}{\kappa}\int_\Omega a(S^k) p_c'(S^k) \nabla S^k \cdot \big( \nabla\beta(S^k)- \nabla \beta(S^{k-1}) \big) dx \nonumber\\
& + \frac{1}{\kappa}\int_\Omega \big( \nabla\beta(S^k)- \nabla \beta(S^{k-1}) \big)  \cdot \nabla \beta(S^k)  \nonumber \\
& + \frac{1}{\kappa^2}\int_\Omega a(S^k) | \nabla \beta(S^k)  -  \nabla \beta(S^{k-1})  |^2 dx  \nonumber \\
& 
 + C\eps \sum_{i=1}^n \|w_i^k\|^2_{H^1(\Omega} 
 + C\sum_{i=1}^n\|\Pi\na \mu_i^k\|_{L^2(\Omega)}^2 \leq 
 \frac{1}{\kappa} \int_\Omega \tilde\Fc(\Sub^{k-1})  dx .
\label{TF1-4}
\end{align}          
Next, using the fact that
$ (a-b)a \geq \frac{1}{2} (a^2 - b^2),$
we obtain
\begin{align*}
 \int_\Omega \nabla \Big(   \frac{\beta(S^k) - \beta(S^{k-1}) }{\kappa} \Big) \cdot \nabla \beta(S^k) dx \geq \frac{1}{2\kappa}  \int_\Omega \Big( |\nabla \beta(S^k)|^2 - |\nabla \beta(S^{k-1})|^2 \Big) dx
\end{align*}
We have:
\begin{align}\nonumber
&\frac{1}{\kappa} \int_\Omega \Big(\tilde\Fc(\Sub^k) 
         + \frac{1}{2} |\nabla \beta(S^k)|^2 \Big) dx 
+ C \int_\Omega\left(\frac{\beta(S^k)-\beta(S^{k-1})}{\kappa}\right)^2 dx\\
& +  \int_\Omega p_c'(S^k)
  \beta'(S^k) |\nabla S^k|^2 dx  \nonumber\\
& + \frac{1}{\kappa} \int_\Omega a(S^k) p_c'(S^k) \nabla S^k \cdot  \Big( \nabla\beta(S^k)-\nabla\beta(S^{k-1}) \Big)  dx  \nonumber \\
& + \frac{1}{\kappa^2} \int_\Omega a(S^k)
 \Big|\nabla\beta(S^k) - \nabla\beta(S^{k-1})  \Big|^2 dx
\nonumber \\
&
+ C\eps \sum_{i=1}^n \|w_i^k\|^2_{H^1(\Omega} 
+ C \sum_{i=1}^n\|\Pi\na \mu_i^k\|_{L^2(\Omega)}^2\nonumber \\
&\qquad \leq \frac{1}{\kappa} \int_\Omega \Big( 
\tilde\Fc(\Sub^{k-1}) +  \frac{1}{2}|\nabla \beta(S^{k-1})|^2  \Big) dx.
\label{TF1-5}
\end{align}     
Next, Young inequality gives:
\begin{align*}
 \int_\Omega a(S^k) &  p_c'(S^k)  
  \nabla S^k \cdot \nabla\frac{\beta(S^k)-\beta(S^{k-1})}{\kappa} dx
\nonumber\\
& \leq \frac{3}{4} \int_\Omega a(S^k) \Big| \nabla  \frac{\beta(S^k) - \beta(S^{k-1})}{\kappa}   \Big|^2 dx \\
&\qquad  + \frac{1}{3} \int_\Omega a(S^k)  (p_c'(S^k))^2 |\nabla S^k|^2 dx
\end{align*}
In this way we get:
\begin{align}\nonumber
\frac{1}{\kappa} \int_\Omega \Big( & \tilde\Fc(\Sub^k) 
        + \frac{1}{2} |\nabla \beta(S^k)|^2 \Big) dx + C \int_\Omega\left(\frac{\beta(S^k)-\beta(S^{k-1})}{\kappa}\right)^2 dx  \\
       & +  \int_\Omega p_c'(S^k) \beta'(S^k) |\nabla S^k|^2 dx  
        + \frac{1}{4\kappa^2} \int_\Omega a(S^k) \Big|\nabla\beta(S^k) - \nabla\beta(S^{k-1})  \Big|^2 dx
\nonumber \\
& + C\eps \sum_{i=1}^n \|w_i^k\|^2_{H^1(\Omega} 
+ C\sum_{i=1}^n\|\Pi\na \mu_i^k\|_{L^2(\Omega)}^2\nonumber \\
& \leq \frac{1}{\kappa} \int_\Omega \Big( \tilde\Fc(\Sub^{k-1}) +  \frac{1}{2}|\nabla \beta(S^{k-1})|^2  \Big) dx
 + \frac{1}{3}  \int_\Omega a(S^k) (p_c'(S^k))^2 |\nabla S^k|^2 dx.
\label{TF1-6}
\end{align}     
Thanks to Lemma~\ref{lemma:coeff:1}, we can estimate the second integral on the right-hand side of \eqref{TF1-6} by means of the
third integral of the left-hand side of \eqref{TF1-6}.
 In this way we get:
\begin{align}\nonumber
\frac{1}{\kappa} \int_\Omega & \Big(\tilde\Fc(\Sub^k) 
         + \frac{1}{2} |\nabla \beta(S^k)|^2 \Big) dx + C \int_\Omega\left(\frac{\beta(S^k)-\beta(S^{k-1})}{\kappa}\right)^2 dx  \\
       & +  \frac{2}{3}\int_\Omega p_c'(S^k) \beta'(S^k) |\nabla S^k|^2 dx  
        + \frac{1}{4\kappa^2} \int_\Omega a(S^k) \Big|\nabla\beta(S^k) - \nabla\beta(S^{k-1})  \Big|^2 dx
\nonumber \\
& + C\eps \sum_{i=1}^n \|w_i^k\|^2_{H^1(\Omega} + C\|\Pi\na \mu_i^k\|_{L^2(\Omega)}^2\nonumber \\
& \leq \frac{1}{\kappa} \int_\Omega \Big(\tilde\Fc(\Sub^{k-1}) +  \frac{1}{2}|\nabla \beta(S^{k-1})|^2  \Big) dx.
\label{TF1-8}
\end{align}     
%
Let us now introduce a new notation. Let us define the piecewise constant-in-time functions:
\begin{align*} 
\Skap_i(t) & = S^0_i\charf_{\{0\}}(t) + \sum_{j=1}^N S^j\charf_{(t_{j-1},t_j]}(t), \\
\mukap_i(t) &= \mu^0_i\charf_{\{0\}}(t) + \sum_{j=1}^N \mu^j\charf_{(t_{j-1},t_j]}(t), 
\end{align*}
and let $\Skap = \sum_{i=1}^n \Skap_i$. We also define the discrete backward time derivative operator $D_\kappa$ as follows:
for every function $f : Q_T\to \R$,
$$ D_\kappa f(x,t) = \frac{f(x,t)-f(x,t-\kappa)}{\kappa}\qquad x\in\Omega,\quad t\in [\kappa, T] . $$
 The discretized-regularized system \eqref{DisReg.1} can be rewritten, in the new notation, as
 \begin{align}\nonumber
  \sum_{i=1}^n \int_0^T\int_\Omega \Big[ \big( D_\kappa \Skap_i  \big)\phi_i &+  \frac{\Skap_i}{\Skap} \frac{a(\Skap)}{\tau(\Skap)} p_c'(\Skap)\nabla \beta(\Skap) \cdot \nabla \phi_i  \\
 & +
 \frac{\Skap_i}{\Skap} a(\Skap) \nabla D_\kappa \beta(\Skap) \cdot \nabla \phi_i   \Big] dx dt \nonumber\\
 & + \int_0^T \int_\Omega \sum_{i,j=1}^n D_{ij}(\Skap_1,\ldots,\Skap_n)\nabla \mu_j^{(\kappa)} \cdot \nabla \phi_i dx dt \nonumber\\
 & + \eps \int_0^T {\int_\Omega}\sum_{i=1}^n \frac{\Skap_i}{\Skap} \nabla \wkap_i \cdot \nabla \phi_i dx dt = 0,\nonumber\\
 &\forall \phi_1,\ldots,\phi_n\in L^2(0,T; H^1_0(\Omega)).
 \label{DisReg.tau.2.1}
 \end{align}
 In the new notation, the entropy inequality \eqref{TF1-8} reads as
\begin{align}\nonumber
&\sup_{t\in [0,T]}\int_\Omega (\tilde\Fc(\Skap)+ \frac{1}{2}|\na\beta(\Skap)|^2) dx \\
& + C \int_0^T\int_\Omega\left(D_\kappa\beta(\Skap)\right)^2 dx dt
+ \frac{2}{3}\int_0^T\int_\Omega \tau(\Skap) p_c'(\Skap)|\nabla \Skap|^2 dx dt
\nonumber\\
& + \frac{1}{4}\int_0^T\int_\Omega a(\Skap)\left|\nabla D_\kappa\beta(\Skap)\right|^2 dx dt
+ C\eps \sum_{i=1}^n \int_0^T\|\wkap_i\|^2_{H^1(\Omega)} dt\nonumber \\
& + C\int_0^T\|\Pi\na\mukap\|_{L^2(\Omega)}^2 dt 
\leq \int_\Omega (\tilde\Fc(\Sub^{0}) 
+ \frac{1}{2}|\na\beta(S^{0})|^2) dx .
\label{dei}
\end{align}            
By using the lower bounded $\di  \sum_{i=1}^n S_i \log \frac{S_i}{S} \geq -m n S \geq -mn$, we obtain the following apriori estimates:

\begin{proposition}\label{prop.7} 
There is a constant $C$, independent of $\kappa$ and $\varepsilon$, 
such that
\begin{align}
	{ \| \mathcal{E}(\Skap) \|_{L^{\infty}(0,T;L^1(\Omega))} } & \leq C, \label{Apr.0}\\
 \| \nabla \beta(\Skap) \|_{L^{\infty}(0,T;L^2(\Omega))} & \leq C, \label{Apr.1}\\
  \| D_\kappa \beta(\Skap) \|_{L^{2}(0,T;L^2(\Omega))} & \leq C, \label{Apr.2}\\
   \| \sqrt{\tau(\Skap) p_c'(\Skap)} \nabla \Skap \|_{L^{2}(0,T;L^2(\Omega))} & \leq C, \label{Apr.3}\\
   \| \sqrt{a(\Skap)} \nabla  D_\kappa \beta(\Skap) \|_{L^{2}(0,T;L^2(\Omega))} & \leq C, \label{Apr.4}\\
  \| \sqrt{\varepsilon}  w_i^{\kappa} \|_{L^{2}(0,T;H^1(\Omega))} & \leq C, \label{Apr.5}\\
   \| \nabla(\Pi \mub^{\kappa})_i \|_{L^2(0,T;L^2(\Omega))} &\leq C,
    \label{Apr.6}
\end{align} 
for $i=1,\ldots,n$.
\end{proposition}
By using the bound \eqref{Apr.0} on the entropy function we obtain the following bounds:
\begin{lemma}
\label{lemma:first-bound}
There is a constant $C$ 
independent of $\kappa$ and $\varepsilon$, such that 
	\begin{align}
		\| (\Skap)^{2-\gamma_1} \|_{L^{\infty}(0,T;L^1(\Omega))}+
		\| (1-\Skap)^{2-\lambda} \|_{L^{\infty}(0,T;L^1(\Omega))}    & \leq C.
		\label{Apr.7}
	\end{align}
\end{lemma}
\begin{proof}
	By using simple calculations, we get
	\begin{align*}
		\mathcal{E}(S) &= \frac{1}{(\gamma_1-1)(\gamma_1-2)} \frac{1}{S^{\gamma_1-2}}
		+  \frac{1}{(\lambda-1)(\lambda-2)} \frac{1}{(1-S)^{\lambda-2}} + \frac{1-S}{\lambda -1} + \frac{1}{\lambda-2}\\
		&\geq C\left( \frac{1}{S^{\gamma_1-2}} + \frac{1}{(1-S)^{\lambda-2}}\right).
	\end{align*}
	The bound (\ref{Apr.7}) now follows from (\ref{Apr.0}).
\end{proof}

By using (\ref{Apr.3}) we get the following bounds.
\begin{lemma} \label{lemma:second-bound}
Define the exponents $\alpha_1$ and $\alpha_2$ as follows:
  \begin{equation}
   \alpha_1 = 1+(\gamma -\gamma_1 -\beta_1)/2 < 0,\quad \alpha_2 = 1-\beta_2/2 <0.
    \label{coeff:assum:4}
  \end{equation}
  Then, there is a constant $C$, independent of $\kappa$ and $\varepsilon$, such that:
  \begin{align}
   	\| (\Skap)^{\alpha_1} \|_{L^{2}(0,T;L^6(\Omega))}+
	\| (1-\Skap)^{\alpha_2} \|_{L^{2}(0,T;L^6(\Omega))} & \leq C.
	\label{Apr.10}
  \end{align}
\end{lemma}
\begin{proof}
    Let us denote,  for notational simplicity,  $S=S^{(\kappa)}$. Then, from (\ref{Apr.3}) and Assumptions {\bf (H3)}, {\bf (H4)} we get
  \begin{align*}
    \int_0^T\int_{\Omega} (S^\gamma + S^{\gamma -\gamma_1}(1-S)^\lambda )( S^{-\beta_1} + (1-S)^{-\beta_2}) |\nabla S|^2 dx dt \leq C,
  \end{align*}
 with the constant $C$ independent of $\kappa$ and $\varepsilon$. 
 As a consequence
  \begin{align*}
    \int_0^T\int_{\Omega}  S^{\gamma -\gamma_1-\beta_1}(1-S)^\lambda |\nabla S|^2 dx dt +
    \int_0^T\int_{\Omega} S^\gamma  (1-S)^{-\beta_2} |\nabla S|^2 dx dt \leq C.
  \end{align*}
  The inequality stated above implies the following bound for the 
  functions $Z\equiv\min(S,1/2)$, $W\equiv\max(S,1/2)$:
  \begin{align*}
    \int_0^T\int_{\Omega}  Z^{\gamma -\gamma_1-\beta_1} |\nabla Z|^2 dx dt +
    \int_0^T\int_{\Omega}  (1-W)^{-\beta_2} |\nabla W|^2 dx dt \leq C,
  \end{align*}
  which can be written as 
  \begin{align*}
    \int_0^T\int_{\Omega}  |\nabla Z^{\alpha_1}|^2 dx dt +
    \int_0^T\int_{\Omega}   |\nabla (1- W)^{\alpha_2}|^2 dx dt \leq C.
  \end{align*}
  The function $Z^{\alpha_1}$ and $ (1-W)^{\alpha_2}$ are in $L^2(0,T;L^2(\Omega))$ due to Assumption {\bf (H5)}
  and Lemma~\ref{lemma:first-bound}. Indeed, Assumption {\bf (H5)} implies that 
  $2\alpha_1 \geq 2- \gamma_1$ and $2\alpha_2 \geq 2-\lambda$. 
  We can then use the Sobolev embedding theorem to get the bound:
 \begin{align*}
   	\| Z^{\alpha_1} \|_{L^{2}(0,T;L^6(\Omega))}+ \| (1-W)^{\alpha_2} \|_{L^{2}(0,T;L^6(\Omega))}    & \leq C.
  \end{align*}
  Due to (\ref{coeff:assum:4}) these bounds hold also for the function $S$ instead of $Z$ and $W$.
\end{proof}

\begin{lemma}
There exists $p>1$ such that 
  \begin{equation}
  \int_0^T \int_\Omega a(S^{(\kappa)})^{-p} dx dt \leq C, \label{Apr.11}
  \end{equation}
 where the constant $C>0$ is   independent of $\kappa$ and $\varepsilon$.
 \label{lemma:third-bound}
\end{lemma}

\begin{proof}
 For simplifying the notation we will write $S=S^{(\kappa)}$. We first notice that 
  \begin{align*}
    a(S)^{-p} = \left[ \frac{1}{S^\gamma} + \frac{1}{(1-S)^\lambda}\right]^p
    \leq C \left[ \frac{1}{S^{p\gamma}} + \frac{1}{(1-S)^{p\lambda}}\right].
  \end{align*}
  So it is sufficient to prove that 
  $S^{-p\gamma}$ and $(1-S)^{-p\lambda}$ 
  are uniformely bounded in $L^1(\Omega\times (0,T))$
  for some $p>1$. 

  It is clear that integrability given by Lemma~\ref{lemma:first-bound} is not 
  sufficient  to prove the estimate (\ref{Apr.11}). Therefore, we will combine estimates from Lemmas~\ref{lemma:second-bound}
  and \ref{lemma:first-bound} in order to obtain the integrability with requested exponents.
Assumptions {\bf (H5)} on the parameters 
$\beta_1$, $\beta_2$, $\gamma$, $\gamma_1$ and $\lambda$ imply
  \begin{equation}
    2<\beta_1\leq \gamma_1 <\gamma < \beta_1 + \gamma_1 -2, \quad 2 < \beta_2 \leq \lambda.
    \label{coeff:ineq:0}
  \end{equation}
 We rewrite the expression $\int_\Omega S^{-\gamma p} dx$ 
 using
 $ -\gamma p = \alpha_1 \Theta + (2-\gamma_1) \Theta_1$
 and H\"older's inequality:
\begin{align*}
 \int_\Omega S^{-\gamma p} dx & = \int_\Omega S^{\alpha_1 \Theta} \; S^{(2-\gamma_1)\Theta_1} dx 
    \leq  \Big( \int_\Omega  S^{\alpha_1 \Theta p_1} dx \Big)^{\frac{1}{p_1}} \; 
	\Big( \int_\Omega S^{(2-\gamma_1)\Theta_1 p_2 } dx \Big)^{\frac{1}{p_2}}.
\end{align*}
We take $p_1 = 6/\Theta$ and $p_2 = 6/(6-\Theta)$, $\Theta = 2$, $\Theta_1 = 2/3$ and we get
\begin{align*}
 \iint_{Q_T} S^{-\gamma p} dx dt  &\leq \int_0^T 
   \Big( \int_{\Omega} S^{6\alpha_1} dx \Big)^{1/3} dt \cdot 
   \max_{0\leq t \leq T} \Big( \int_{\Omega} S^{2-\gamma_1} dx \Big)^{2/3}\\
   &=\| S^{\alpha_1}\|_{L^2(0,T; L^6(\Omega))}^2 \| S^{2-\gamma_1}\|_{L^\infty(0,T; L^1(\Omega))}^{2/3}.
\end{align*}
Because of \eqref{Apr.10} and  \eqref{Apr.7}, the right hand side is uniformly bounded. Condition
\begin{align*}
 p = - \frac{1}{\gamma} \Big( \frac{10}{3} + \gamma - \frac{5}{3} \gamma_1 - \beta_1 \Big) > 1 
\end{align*}
is equivalent to
\begin{align}
  \gamma < \frac{1}{2} \beta_1 +  \frac{5}{6}(\gamma_1 -2).
 \label{Est.n3.1}
\end{align}
Now it is easy to see that \eqref{Est.n3.1} and the first inequality in \eqref{coeff:ineq:0} are equivalent to the first
inequality in Assumption {\bf (H5)}.

The second inequality in Assumption {\bf (H5)} in treated in the same way. 
The calculations are given here for completeness.
 We rewrite the expression $\int_\Omega (1-S)^{-\lambda p} dx$ 
 using
 $ -\lambda p = \alpha_2 \Theta + (2-\lambda) \Theta_1$
 and H\"older's inequality:
\begin{align*}
 \int_\Omega (1-S)^{-\lambda p} dx & = \int_\Omega (1-S)^{\alpha_2 \Theta} \; (1-S)^{(2-\lambda)\Theta_1} dx\\ 
    &\leq  \Big( \int_\Omega  (1-S)^{\alpha_2 \Theta p_1} dx \Big)^{\frac{1}{p_1}} \; 
	\Big( \int_\Omega (1-S)^{(2-\lambda)\Theta_1 p_2 } dx \Big)^{\frac{1}{p_2}}.
\end{align*}
We take $p_1 = 6/\Theta$ and $p_2 = 6/(6-\Theta)$, $\Theta = 2$, $\Theta_1 = 2/3$ and obtain
\begin{align*}
 \iint_{Q_T} (1-S)^{-\lambda p} dx dt  &\leq \int_0^T 
   \Big( \int_{\Omega} (1-S)^{6\alpha_2} dx \Big)^{1/3} dt \cdot 
   \max_{0\leq t \leq T} \Big( \int_{\Omega} (1-S)^{2-\lambda} dx \Big)^{2/3}\\
   &=\| (1-S)^{\alpha_2}\|_{L^2(0,T; L^6(\Omega))}^2 \| (1-S)^{2-\lambda}\|_{L^\infty(0,T; L^1(\Omega))}^{2/3}.
\end{align*}
Because of \eqref{Apr.10} and  \eqref{Apr.7}, the right hand side is uniformly bounded. Condition 
\begin{align*}
 p = - \frac{1}{\lambda} \Big( \frac{4}{3} - \frac{2}{3} \lambda  + 2- \beta_2 \Big) > 1 
\end{align*}
is equivalent to $\lambda < 3\beta_2 -10$. It is now easy to see that this inequality together with the second inequality in 
\eqref{coeff:ineq:0} are equivalent to the second
inequality in Assumption {\bf (H5)}. 
This concludes the proof of Lemma~\ref{lemma:third-bound}.
\end{proof}

\begin{proposition}\label{prop.12}
  There is an exponent $1<q<2$ such that 
  \begin{equation}
    \| \nabla  D_\kappa \beta(\Skap)\|_{L^q(\Omega\times (0,T)))} \leq C,
    \label{bound:3rdder}
  \end{equation}
  where $C$ is a constant   independent of $\kappa$ and $\varepsilon$.
\end{proposition}
\begin{proof}
  Let $p>1$ as in Lemma~\ref{lemma:third-bound}. By choosing $q=2p/(1+p) \in (1,2)$, we get 
  \begin{align*}
    &\int_0^T \int_\Omega |\nabla  D_\kappa \beta(\Skap)|^q dx dt \\
    &=
    \int_0^T \int_\Omega {a(\Skap)}^{-q/2} {a(\Skap)}^{q/2}|\nabla  D_\kappa \beta(\Skap)|^q dx dt\\
    &\leq \left( \int_0^T \int_\Omega {a(\Skap)}^{-q/(2-q)} dx dt\right)^{(2-q)/2} 
      \left( \int_0^T \int_\Omega {a(\Skap)} |\nabla  D_\kappa \beta(\Skap)|^2 dx dt\right)^{q/2}\\
      &=  \left( \int_0^T \int_\Omega {a(\Skap)}^{-p} dx dt\right)^{(2-q)/2} 
      \left( \int_0^T \int_\Omega {a(\Skap)} |\nabla  D_\kappa \beta(\Skap)|^2 dx dt\right)^{q/2} .
  \end{align*}
  By using Lemma~\ref{lemma:third-bound} and bound (\ref{Apr.4}) we conclude the proof. 
\end{proof}

Finally, from equation \eqref{DisReg.tau.2.1} together with the the bounds \eqref{Apr.3}--\eqref{Apr.6}, we get the following uniform bound for the discrete time derivative:
\begin{align}\label{dicret.time.deriv}
 \|D_\kappa \Skap_i\|_{L^2(0,T; H^{-1}(\Omega))}\leq C_T.
\end{align}

\subsection{Passing to the limit when $\kappa \to 0$}

From \eqref{Apr.5} and Lemma \ref{lemma.3} we get that
\begin{align*}
\|\sqrt{\varepsilon}S_i^{(\kappa)}\|_{L^2(0,T; H^1(\Omega))} \leq C_T.
\end{align*}
From this and the bound from the discrete time derivative \eqref{dicret.time.deriv} we get by using the nonlinear version of the Aubin-Lions lemma \cite{CJL14} that
\begin{align*}
S_i^{\varepsilon,\kappa} \to S_i^{\varepsilon} \quad \mbox{strongly in}~~ L^2(0,T; L^2(\Omega)).
\end{align*}
This strong convergence holds also in $L^q(0,T; L^q(\Omega))$ for any $q<\infty$.

By using the bounds in Proposition \ref{prop.7} and Proposition \ref{prop.12}, we obtain that
the solution of \eqref{DisReg.tau.2.1} satisfies
 \begin{align}\nonumber
 \sum_{i=1}^n \int_0^T \langle \pa_t S^{(\varepsilon)}_i,
 \phi_i\rangle dt
 &+\sum_{i=1}^n \int_0^T\int_\Omega \Big[\frac{S^{(\varepsilon)}_i}{S^{(\varepsilon)}} \frac{a(S^{(\varepsilon)})}{\tau(S^{(\varepsilon)})} p_c'(S^{(\varepsilon)})\nabla \beta(S^{(\varepsilon)}) \cdot \nabla \phi_i  \\
  & +
  \frac{S^{(\varepsilon)}_i}{S^{(\varepsilon)}} a(S^{(\varepsilon)}) \nabla D_\kappa \beta(S^{(\varepsilon)}) \cdot \nabla \phi_i   \Big] dx dt \nonumber\\
  & + \int_0^T \int_{\Omega} \sum_{i,j=1}^n D_{ij}(S_i^{(\varepsilon)},\ldots,S_n^{(\varepsilon)})\nabla \mu_j^{(\kappa)} \cdot \nabla \phi_i dx dt \nonumber\\
  & + \eps \int_0^T \int_\Omega\sum_{i=1}^n \frac{S^{(\varepsilon)}_i}{S^{(\varepsilon)}} \nabla \wkap_i \cdot \nabla \phi_i dx dt = 0.
  \label{DisReg.tau.2.2}
 \end{align}
Thus, after taking the limit $\kappa\to 0$,  
\eqref{Apr.11} holds with $S^{(\varepsilon)}$ in place of $S^{(\kappa)}$, \textit{i.e.}
\begin{align}\label{est.75new}
 \int_0^T \int_\Omega a(S^{(\varepsilon)})^{-p} dx dt \leq C.
\end{align}
Also, estimates \eqref{Apr.0}-\eqref{Apr.6}, \eqref{bound:3rdder},
\eqref{dicret.time.deriv} hold with 
$S^{(\varepsilon)}$ in place of $S^{(\kappa)}$, \textit{i.e.}
\begin{align}
	{ \| \mathcal{E}(S^{(\varepsilon)}) \|_{L^{\infty}(0,T;L^1(\Omega))} } & \leq C, \label{Apr.0.new}\\
 \| \nabla \beta(S^{(\varepsilon)}) \|_{L^{\infty}(0,T;L^2(\Omega))} & \leq C, \label{Apr.1.new}\\
  \| \pa_t \beta(S^{(\varepsilon)}) \|_{L^{2}(0,T;L^2(\Omega))} & \leq C, \label{Apr.2.new}\\
   \| \sqrt{\tau(S^{(\varepsilon)}) p_c'(S^{(\varepsilon)})} \nabla S^{(\varepsilon)} \|_{L^{2}(0,T;L^2(\Omega))} & \leq C, \label{Apr.3.new}\\
   \| \sqrt{a(S^{(\varepsilon)})} \nabla  \pa_t  \beta(S^{(\varepsilon)}) \|_{L^{2}(0,T;L^2(\Omega))} & \leq C, \label{Apr.4.new}\\
  \| \sqrt{\varepsilon}  w_i^{\varepsilon} \|_{L^{2}(0,T;H^1(\Omega))} & \leq C, \label{Apr.5.new}\\
   \| \nabla(\Pi\mub^{\varepsilon})_i \|_{L^2(0,T;L^2(\Omega))} 
   &\leq C,    \label{Apr.6.new}\\
   \|\pa_t S_i^{(\eps)}\|_{L^2(0,T; H^{-1}(\Omega))} &\leq C,
   \label{9.1.19.dt}\\
   \|\nabla\pa_t\beta(S^{(\eps)})\|_{L^q(\Omega\times (0,T))}
   &\leq C, \label{9.1.19.dd}
\end{align}
for $i=1,\ldots,n$.

 \subsection{Passing to the limit $\varepsilon \to 0$}

Now we define the continuous mapping 
$\mathcal{R}: \R_+\times\R^{n}\to\R^{n+1}$ as
\[ \mathcal{R}_i(w_0,\overline{w}) = w_0 \frac{e^{\overline{w}_i}}{\sum_{j=1}^n e^{\overline{w}_j}},\qquad 
w_0\geq 0,~~\overline{w}\in\R^{n},~~ i=1,\ldots,n.\]
It follows from \eqref{Si.2} that
$S_i^{(\eps)} = \mathcal{R}_i( S^{(\eps)}, (\mu^*)^{(\eps)})$
for $i = 1,\ldots, n$. Lemma~\ref{Dreyer.2} implies that 
$S_i^{(\eps)}$ has a subsequence that
is strongly convergent in $L^1(\Omega \times (0,T))$, 
for $i = 1,\ldots, n$. 
From the $L^\infty$ bounds for $S_i^{(\eps)}$ 
we conclude that, up to
a subsequence, it holds that
\begin{align*}
 S_i^{(\varepsilon)} \to S_i \quad \mbox{strongly in }L^q(\Omega \times (0,T)) \quad \mbox{for all }q<\infty,~~ i = 1,\ldots, n.
\end{align*}
By using this convergence property as well as the bounds \eqref{est.75new}--\eqref{9.1.19.dd}, we are able to take the limit $\varepsilon \to 0$ in \eqref{DisReg.tau.2.2} 
and obtain that $\Sub = (S_1,\ldots,S_n)$ is a weak solution to 
\eqref{1}--\eqref{IC}.
This finishes the proof of the theorem.

\section*{Appendix}

\subsection*{Derivation of the model}

 We consider an isothermal, immiscible and incompressible two-phase flow of water and oil in a porous media, where oil consists of $n$ chemical components. Let us denote by $\mathcal{V}$ the representative volume (REV), which consists of the solid part $\mathcal{V}_s$ and the
 pore space $\mathcal{V}_p$. The flow occurs in a porous domain $\mathcal{V}_p$ of volume $\Delta V_p$, where the porosity (the relative volume occupied by the pores) is denoted by $\di \Phi = \Delta V_p/\Delta V$. 
 The saturations of the oil and water phase are given by
$  S_\alpha = \Delta V_\alpha/\Delta V_p$,
where $\Delta V_\alpha$ is the volume of the phase $\alpha$ with $\alpha=w,o$.
 Following \cite{BearBach90}, a generalized Darcy law gives
  \begin{align}
   \ub_w = -\lambda_w(S_o) k \nabla p_w,\quad  \ub_o = -\lambda_o(S_o) k \nabla p_o.
   \label{Darcy}
 \end{align}
Here the subscripts $w$ and $o$ correspond, respectively, to the water (wetting) and the oil (non-wetting) fluids,
$\ub_\alpha$ are the fluxes of the phases, $p_\alpha$ are their pressures, and $\lambda_\alpha$ are the phase mobilities. We assume that $\lambda_\alpha$ depend on the
nonwetting-phase saturation $S_o$.
Furthermore, $k$ is the absolute permeability of the porous medium,
and the gravity effects are neglected for simplicity.
 The mass conservation laws for both phases have the form:
\begin{align}
 \Phi \frac{\rho_o \partial S_o}{\partial t} + \dv \rho_o \ub_o  = 0, \quad
 \Phi \frac{\rho_w \partial(1-S_o)}{\partial t} + \dv  \rho_w \ub_w  = 0, \label{eq-1-2}
\end{align}
where $\Phi$ is the porosity of the medium.
The model \eqref{Darcy}--\eqref{eq-1-2} has to be completed with the capillary pressure law 
which has the form
\[ p_o - p_w = p_c^{\textrm{dyn}}, \]
where, due to \cite{HG93}, the capillary pressure saturation relationship is given by
\begin{align}
 p_c^{\textrm{dyn}} = p_c(S_o) + \tau(S_o) \frac{\partial S_o}{\partial t}.
 \label{HassGrey}
\end{align}
Here, $ p_c(S_o)$ is the static capillary pressure function and $\tau(S_o)$ is the relaxation parameter.

We assume that the non-wetting phase (oil) is a heterogeneous mixture of hydrocarbon compounds and we derive the mass conservation equation for each compound.
 More precisely, in the oil phase there are $n$ components whose
{\sl mass concentrations $c_o^i$}, \textit{i.e.}~the densities of the $i$-th component in the volume of the phase, are given by
$ c_o^i = \Delta m_o^i/\Delta V_o$,
where $\Delta m_o^i$ is the mass of the component $i$ in the oil-phase of the REV.
 The sum of the mass concentrations of all components is given by 
\begin{align}
  \sum_{i=1}^n  c_o^i = \frac{\Delta m_o}{\Delta V_o} = \rho_o.
  \label{Der.1}
\end{align}  
Noting that
\[ \frac{\Delta m_o^i}{\Delta V} =  \frac{\Delta V_p}{\Delta V} \frac{\Delta V_o}{\Delta V_p}  \frac{\Delta m_o^i}{\Delta V_o}=  \Phi S_o c_o^i, \]
 the mass conservation equation for the component $i$ is given by
 \begin{align}
   \frac{\partial}{\partial t}  \big( \Phi S_o c_o^i \big) +  \dv \big( c_o^i  \ub_{o,i}\big) = 0.
 \label{Conser.2}
 \end{align}
The component velocities $\ub_{o,i}$ are related to the phase velocity $\ub_o$ by the expression
 \begin{align}
  \rho_o \ub_o = \sum_{i=1}^n c_o^i \ub_{o,i}.
 \label{CompVel.1}
 \end{align}
 
%

 The flux of the oil-phase components consists of the relative movement of the constituents $i$ spreading due to random collisions between molecules of different types (diffusion) followed by the convection, \textit{i.e.}
 \begin{align}
  c_o^i  \ub_{o,i} =  \Jb_o^i + c_o^i \ub_o.
 \label{Diff.1}
\end{align}  
Note that $\sum_{i=1}^n \Jb_o^i = \Nullb$. 
Let us introduce the saturation of the component $i$ in the oil phase as
\[ S_o^i = \frac{\Delta V_o^i}{\Delta V_p}. \]
It is clear that $\sum_{i=1}^n S_o^i = S_o$. Furthermore, we assume that each component $i$ of the mixture in the oil phase is {\em incompressible}, \textit{i.e.}
\[ \Delta m_o^i = \rho_o^i \Delta V_o^i,\; \mbox{ where }\; \rho_o^i = \mbox{const.}\]
Now we have
\begin{align*}
 c_o^i = \frac{\Delta m_o^i}{\Delta V_o} = \frac{\rho_o^i\, \Delta V_o^i}{\Delta V_o}
    = \rho_o^i \frac{\Delta V_o^i/\Delta V_p}{\Delta V_o/\Delta V_p} = \rho_o^i \frac{S_o^i}{S_o}.
\end{align*}
 Next, we make the assumption that the diffusion fluxes are proportional to the spatial gradients of suitable chemical potentials, \textit{i.e.}
\begin{align}
 \Jb_o^i : = - \rho_o^i  \sum_{j=1}^n D_{ij} (S_o^1,\ldots,S_o^n) \nabla \mu_j, \qquad i=1,\ldots,n,
 \label{Diff.Flux}
\end{align}
where $\mu_j$ are given by \eqref{ChemPot} using the notation $S \equiv S_o$.
In this way, equation \eqref{Conser.2} reads:
\begin{align*}
   \rho_o^i \frac{\partial}{\partial t}  \Big( \Phi S_o \frac{S_o^i}{S_o} \Big) 
       +  \dv \Big( \rho_o^i \frac{S_o^i}{S_o}  \ub_{o} -  \rho_o^i \sum_{j=1}^n  D_{ij} (S_o^1,\ldots,S_o^n)  \nabla \mu_j \Big) = 0.
\end{align*}
%
Now, a simple calculation gives
\begin{align*}
 -\lambda_o k \nabla p_o 
   =  \frac{\lambda_o \lambda_w}{\lambda_o + \lambda_w} k \nabla ( p_w - p_o) + \frac{\lambda_o}{\lambda_o + \lambda_w}(\ub_0 + \ub_w).
\end{align*}

Furthermore, we assume that the total flow equals zero, \textit{i.e.} $\ub_0 + \ub_w = \bf{0}$, which gives   
\begin{align*}
 \lambda_o(S_o) \nabla p_o =  a(S_o) \nabla p_c^{\textrm{dyn}}.
\end{align*}
Here the diffusion mobility $a(S_o)$ is given by
\[  a(S_o) = \frac{\lambda_o(S_o) \lambda_w(S_o)}{\lambda_o(S_o) + \lambda_w(S_o)}.    \]

In this way, we obtain the parabolic system of our interest
\begin{align}
  \partial_t S_o^i - \dv \left(  \frac{S_o^i}{S_o} a(S_o) k \nabla p_c^{\textrm{dyn}} 
  + \sum_{j=1}^n D_{ij}(S_o^1,\ldots,S_o^n) \nabla \mu_j \right) = 0,
 \label{Eq_1}
\end{align}
where $i=1,\ldots,n$. Notice that \eqref{1} is identical to 
\eqref{Eq_1} with $k=1$ and $S_o$, $S_o^i$ replaced by $S$, $S_i$,
respectively.


\end{document}